\documentclass{amsart}

\usepackage{t1enc}
\usepackage[latin2]{inputenc}
\usepackage{amsmath}
\usepackage{amssymb}
\usepackage{amsthm}
\usepackage{eufrak}
\usepackage{verbatim}
\usepackage{color}
\usepackage{pstricks,pst-node}
\usepackage{enumerate}
\usepackage{xmpmulti}
\usepackage{psfrag}
\usepackage{graphicx}

\newtheorem{thm}{Theorem}[section]

\newtheorem{clm}[thm]{Claim}

\author[D. Soukup]{D\'aniel Soukup}
\author[L. Soukup]{Lajos Soukup}
\address{E\"otv\"os L\'or\'and University}
\email{daniel.t.soukup@gmail.com}
\address{Alfr\'ed R\'enyi Institute of Mathematics}
\email{soukup@renyi.hu}
\title{Club guessing for dummies}

\begin{document}
\maketitle

\begin{abstract}
We give a direct, detailed and relatively short proof of Shelah's theorem on club guessing sequences on $S^{\mu^+}_\mu$ (for any regular, uncountable cardinal $\mu$).
\end{abstract}

\section{Introduction}

The aim of this paper is to give an easy proof of a less known club guessing theorem of Shelah. We will use the following notations. For a cardinal $\lambda$ and a regular cardinal $\mu$ let $S^\lambda_\mu$ denote the ordinals in $\lambda$ with cofinality $\mu$. For $S\subseteq S^\lambda_\mu$ an  \emph{$S$-club sequence} is a sequence $\underline{C}=\langle C_\delta:\delta\in S\rangle$ such that $C_\delta \subseteq \delta$ is a club in $\delta$ of order type $\mu$. One of the basic results in Shelah's club guessing theory is the following.

\begin{thm}[{\cite[Claim 2.3]{sh365}}]\label{shelah1}Let $\lambda$ be a cardinal such that $cf(\lambda)\geq \mu^{++}$ for some regular $\mu$ and let $S\subseteq S_\mu^\lambda$ stationary. Then there is an $S$-club sequence $\underline{C}=\langle C_\delta:\delta\in S\rangle$ such that for every club $E\subseteq \lambda$ there is $\delta\in S$ (equivalently, stationary many) such that $C_\delta\subseteq E$.
\end{thm}

Actually Shelah in \cite{sh365} proves more, but a detailed proof of Theorem \ref{shelah1} can be found in \cite[Theorem 2.17]{card}.\\ What can be said about guessing clubs on $S^{\mu^+}_\mu$? It is known that in ZFC there are no such strong guessing sequences generally but some approximation can be done in this case either.

\section{Club guessing on $S^{\mu^+}_\mu$}

\begin{thm}[{\cite[Claim 3.3]{sh572}}] \label{shelah2} Let $\lambda$ be a cardinal such that $\lambda=\mu^+$ for some uncountable, regular $\mu$ and $S\subseteq S^\lambda_\mu$ stationary. Then there is an $S$-club sequence $\underline{C}=\langle C_\delta:\delta\in S \rangle$ such that $C_\delta=\{\alpha^\delta_\zeta:\zeta<\mu\}\subseteq \delta$ and for every club $E\subseteq \lambda$ there is $\delta\in S$ (equivalently, stationary many) such that: $$\{\zeta<\mu:\alpha^\delta_{\zeta+1}\in E\} \text{ is stationary}.$$
\end{thm}
\begin{proof} Let us prove something easier first.
\begin{clm}\label{pim}There is an $S$-club sequence $\underline{C}=\langle C_\delta:\delta\in S\rangle$, such that $C_\delta=\{\alpha^\delta_\zeta:\zeta<\mu\}$ and for every club $E\subseteq \lambda$ there is $\delta\in S$ (equivalently, stationary many) such that $$\{\zeta<\mu:(\alpha^\delta_\zeta,\alpha^\delta_{\zeta+1})\cap E\neq\emptyset\} \text{ is stationary}.$$
\end{clm}
\begin{proof} Suppose on the contrary that this does not hold. Take any $S$-club sequence $\underline{C}_0=\langle C_\delta^0:\delta\in S\rangle$. This does not satisfy the above claim, thus there is some club $E_0\subseteq \lambda$ such that for every $\delta\in S$ there is a club $e_\delta^0\subseteq \delta$ such that $$\alpha\in e_\delta^0\cap C_\delta^0\Rightarrow E_0\cap (\alpha,\min( C_\delta^0\setminus\alpha+1))=\emptyset$$
Let $C_\delta^1=C_\delta^0\cap e_\delta^0$ for $\delta\in S$ and $\underline{C}_1=\langle C_\delta^1:\delta\in S\rangle$. If we defined $\underline{C}_n$ for some $n\in\omega$ then  $\underline{C}_n$ does not guesses well each club, thus there is some club $E_{n}\subseteq \lambda$ such that for every $\delta\in S$ there is a club $e_\delta^n\subseteq \delta$ such that $$\alpha\in e_\delta^n\cap C_\delta^n\Rightarrow E_n\cap (\alpha,\min( C_\delta^n\setminus\alpha+1))=\emptyset$$
Let $C_\delta^{n+1}=C_\delta^n\cap e_\delta^n$ for $\delta\in S$ and $\underline{C}_{n+1}=\langle C_\delta^{n+1}:\delta\in S\rangle$.\\ Since $cf(\delta)=\mu>\omega$ for $\delta\in S$, $C_\delta=\bigcap\{C_\delta^n:n\in\omega\}$ is a club in $\delta$. Let $E=\bigcap\{E_n:n\in\omega\}$ and pick $\delta\in E\cap S$ such that $tp(\delta\cap E) > \mu$. Thus there is some $\xi<\delta$ such that $\xi \in E\setminus C_\delta$ and $\min C_\delta < \xi$. Since the sequence $\langle C_\delta^n:n\in\omega\rangle$ is decreasing, $\sup (C_\delta^n\cap\xi)$ is decreasing either thus there is some $m\in \omega$ such that $\alpha=\sup (C_\delta^n\cap\xi)<\xi$ for $n\geq m$. Since $\alpha=\sup (C_\delta^{n+1}\cap\xi)\in C_\delta^{n+1}= e_\delta^n\cap C_\delta^n$  and $\xi\in E_n$, $\alpha<\xi<\min( C_\delta^n\setminus\alpha+1)$, this contradicts the fact that $E_n\cap (\alpha,\min( C_\delta^n\setminus\alpha+1))=\emptyset$.
\end{proof}
Let $\langle C_\delta:\delta\in S\rangle$ be an $S$-club sequence given by Claim \ref{pim}, with continuous enumeration $C_\delta=\{\alpha^\delta_\zeta:\zeta<\mu\}$. For any club $E\subseteq \lambda$ and $\delta\in S$ define the following partial function $f_{\delta, E}$ on $\mu$:
$$f_{\delta,E}(\zeta)=\sup \bigl(E\cap(\alpha^\delta_\zeta,\alpha^\delta_{\zeta+1})\bigr)$$
Claim \ref{pim} states that for any club $E\subseteq \lambda$ there is stationary many $\delta\in S$ such that $\text{dom} f_{\delta,E}$ is stationary in $\mu$.
\begin{clm}\label{dana} There is a club $E_0\subseteq \lambda$ such that for each club $E\subseteq E_0$ there is $\delta \in S$ (equivalently, stationary many) such that $$\{\zeta<\mu:\zeta\in \text{dom}f_{\delta, E}\text{ and } f_{\delta, E}(\zeta)=f_{\delta, E_0}(\zeta)\} \text{ is stationary}.$$
\end{clm}
\begin{proof}Suppose on the contrary that this does not hold. Thus $E_0=\lambda$ is not good, so there is some club $E_1\subseteq E_0$ such that for every $\delta\in S$ there is some club $d_\delta^1\subseteq \mu$ such that$$\zeta\in d_\delta^1 \Rightarrow \zeta\notin \text{dom} f_{\delta,E_1} \text{ or } f_{\delta,E_1}(\zeta)<f_{\delta,E_0}(\zeta).$$
If we have the club $E_n$ for some $n\in\omega$, then there is some club $E_{n+1}\subseteq E_n$ such that for every $\delta\in S$ there is some club $d_\delta^{n+1}\subseteq \mu$ such that$$\zeta\in d_\delta^{n+1} \Rightarrow \zeta\notin \text{dom} f_{\delta,E_{n+1}} \text{ or } f_{\delta,E_{n+1}}(\zeta)<f_{\delta,E_n}(\zeta).$$
Let $E=\bigcap\{E_n:n\in \omega\}$, then $E$ is a club. There is stationary many $\delta\in S$ such that $\text{dom} f_{\delta,E}$ is stationary, let $\delta \in E$ such that $\text{dom} f_{\delta,E}$ is stationary. Since for $n\in\omega:$ $E\subseteq E_n$ thus $\text{dom} f_{\delta,E}\subseteq \text{dom} f_{\delta,E_n}$. Let $d=\bigcap\{d_\delta^n:n\in\omega\}$, then $d\subseteq \mu$ is a club. Thus there is some $\zeta\in d\cap \text{dom} f_{\delta,E}$, thus $\zeta\in \text{dom} f_{\delta,E_n}$ for each $n\in\omega$. But then by the definition of the sets $d_\delta^n$ we have an infinite decreasing sequence of ordinals: $$f_{\delta,E_0}(\zeta)>f_{\delta,E_1}(\zeta)>...>f_{\delta,E_n}(\zeta)>f_{\delta,E_{n+1}}(\zeta)>...$$ which is a contradiction.
\end{proof}
With the aid of this club $E_0$ we modify the sequence $\langle C_\delta:\delta\in S\rangle$. Let $S_0=\{\delta\in S: \text{dom}f_{\delta,E_0} \text{is stationary in }\mu\}\subseteq S$, then $S_0$ is stationary. Let $\tilde{C}_\delta=C_\delta$ for $\delta\in S\setminus S_0$. Let $\tilde{C}_\delta=C_\delta\cup\{f_{\delta,E_0}(\zeta):\zeta\in \text{dom}f_{\delta,E_0}\}$, clearly $\tilde{C}_\delta$ is a club in $\delta$. We claim that the sequence $\langle \tilde{C}_\delta:\delta\in S\rangle$ has the desired property. Let $D\subseteq \lambda$ be any club, let $E=D\cap E_0\subseteq E_0$. Then by Claim \ref{dana} for stationary many $\delta\in S$ $\{\zeta<\mu:\zeta\in \text{dom}f_{\delta, E}\text{ and } f_{\delta, E}(\zeta)=f_{\delta, E_0}(\zeta)\}$ is stationary. For such a $\zeta<\mu$, the successor of $\alpha^\delta_\zeta$ in $\tilde{C}_\delta$ is $f_{\delta, E_0}(\zeta)$ and $f_{\delta,E_0}(\zeta)\in D$, since $f_{\delta, E_0}(\zeta)=f_{\delta, E}(\zeta)\in E=E_0\cap D$.
\end{proof}

\end{document}